\documentclass[11pt]{article}

\usepackage[utf8]{inputenc} 
\usepackage[T1]{fontenc}    
\usepackage{url}            
\usepackage{booktabs}       
\usepackage{amsfonts}       
\usepackage{nicefrac}       
\usepackage{microtype}      
\usepackage{enumitem}
\usepackage{multirow}
\usepackage{array}
\usepackage[noblocks]{authblk}

\usepackage[font=small,labelfont=bf]{caption}

\usepackage{tcolorbox}
\usepackage{xcolor}
\usepackage[colorlinks = true,
            linkcolor = blue,
            urlcolor  = blue,
            citecolor = blue,
            anchorcolor = blue]{hyperref}

\usepackage[margin=1in]{geometry}
\makeatletter
\newcommand*{\centerfloat}{%
  \parindent \z@
  \leftskip \z@ \@plus 1fil \@minus \textwidth
  \rightskip\leftskip
  \parfillskip \z@skip}
\makeatother

\usepackage{caption}
\usepackage[numbers]{natbib}
\usepackage{amsthm}
\usepackage{amsmath,amssymb}
\usepackage{graphicx,color}
\usepackage{subfigure,setspace}

\newcommand{\tr}{{{\mathsf T}}}

\setlist{leftmargin=0pt,itemindent=15pt,labelwidth=8pt,labelsep=7pt,listparindent=15pt}

\title{\bfseries Non-asymptotic Identification of Linear Dynamical Systems Using Multiple Trajectories\thanks{This work is supported by NSF career 1553407, AFOSR Young Investigator Program, and ONR Young Investigator Program. Emails: zhengy@g.harvard.edu; nali@seas.harvard.edu.}
}

\author[1,2]{Yang Zheng}
\author[1,2]{Na Li}
\affil[1]{\small School of Engineering and Applied Sciences,
Harvard University}
\affil[2]{\small Harvard Center for Green Buildings and Cites, Harvard University}

\newtheorem{theorem}{Theorem}[section]
\newtheorem{proposition}{Proposition}[section]
\newtheorem{lemma}{Lemma}[section]
\newtheorem{corollary}{Corollary}[section]

\newtheorem{remark}{Remark}

\begin{document}

\maketitle

\begin{abstract}
This paper considers the problem of linear time-invariant (LTI) system identification using input/output data. Recent work has provided non-asymptotic results on partially observed LTI system identification using a single trajectory but is only suitable for stable systems. We provide finite-time analysis for learning Markov parameters based on the ordinary least-squares (OLS) estimator using multiple trajectories, which covers both stable and unstable systems. For unstable systems, our results suggest that the Markov parameters are harder to estimate in the presence of process noise. Without process noise, our upper bound on the estimation error is independent of the spectral radius of system dynamics with high probability. These two features are different from fully observed LTI systems for which recent work has shown that unstable systems with a bigger spectral radius are easier to estimate. Extensive numerical experiments demonstrate the performance of our OLS estimator.
\end{abstract}

\section{Introduction}

System identification estimates the models of dynamical systems from observed input-output data~\cite{ljung1999system}, which is an important topic in  time-series analysis, control theory, robotics, and reinforcement learning. There is an extensive literature on theoretical and algorithmic developments of system identification, with many excellent textbooks~\cite{ljung1999system,van2012subspace} and surveys~\cite{aastrom1971system,pillonetto2014kernel,qin2006overview} available. Classical results often offer asymptotic convergence guarantees for learning system models from observed data~\cite{ljung1999system,qin2006overview}. There has been an
increasing interest in finite sample complexity and non-asymptotic analysis, since good error bounds are essential for designing high-performance robust control systems as well as for establishing end-to-end performance guarantees~\cite{dean2017sample,tu2017non,zhou1996robust}.

In this paper, we consider the problem of identifying a discrete-time linear time-invariant (LTI) system
\begin{equation} \label{eq:dynamics}
    \begin{aligned}
        x_{t+1} = A x_{t} + Bu_t + B_w w_t\\
        y_t     = C x_t + Du_t + D_v v_t,
    \end{aligned}
\end{equation}
 where  $x_t \in \mathbb{R}^{n}$, $y_t \in \mathbb{R}^{p}$, $u_t \in \mathbb{R}^{m}$ are the system state, output and input at time $t$, respectively, and $w_t \in \mathbb{R}^{q}, v_t \in \mathbb{R}^{l}$ denote the process and measurement noises. When $C = I, D = 0$ and $D_v = 0$, \emph{i.e.}, we have direct state observations,~\eqref{eq:dynamics} is known as a \emph{fully observed LTI system}, for which a line of recent work has derived sharp finite-time error analysis based on ordinary least-squares~(OLS) using a single trajectory~\cite{sarkar2018near,simchowitz2018learning}. Notably, despite the correlation between states $x_0, x_1, \ldots, x_T$ and the noise $w_0, w_1, \ldots, w_{T-1}$,
the analysis technique~\cite{sarkar2018near} is suitable for both~stable and unstable systems (captured by $\rho(A) < 1$, $\rho(A) \geq 1$, respectively --- the spectral radius of $A$), under a regularity condition. Furthermore, it is shown that \emph{more unstable} linear systems with a bigger $\rho(A)$ are \emph{easier} to estimate using OLS~\cite{simchowitz2018learning}. This is consistent with the result in~\cite{dean2017sample}, where only the last data sample of multiple independent trajectories is used.

When $C\neq I$, \emph{i.e.},~\eqref{eq:dynamics} is a \emph{partially observed LTI system}, OLS can also be used to recover Markov parameters which in turn allow constructing the system parameters~$A, B$, $C, D$ using the celebrated Ho-Kalman algorithm~\cite{ho1966effective}. However, the methods and analysis techniques are typically more involved than those for fully observed LTI systems. In~\cite{oymak2019non,sarkar2019finite}, OLS is used to estimate Markov parameters and the Hankel matrix using a single trajectory, respectively; but their methods require strictly stable system $\rho(A) < 1$ and non-asymptotic bounds depend on the inverse stability gap $\frac{1}{1 - \rho(A)}$. More recently, the method in~\cite{oymak2019non} has been extended using a pre-filtered variation of the OLS in~\cite{simchowitz2019learning}, working for marginally stable system $\rho(A) \leq 1$, and the finite-time estimation error therein does not degenerate as $\rho(A) \rightarrow 1$. Using outputs as regressors also allows for stochastic system identification with $\rho(A) \leq 1$~\cite{tsiamis2019finite}. In addition to OLS, projected gradient descent can recover the state-space representation~\eqref{eq:dynamics} using a polynomial number of samples~\cite{hardt2018gradient}, but this method requires an assumption stronger than strictly stability of $A$.

One key factor preventing the techniques in~\cite{oymak2019non,sarkar2019finite,simchowitz2019learning} from unstable systems is that the error state $e_t  = CA^{T-1}x_{t-T+1}$ becomes unbounded if $\rho(A) >1$, where $T$ is the length of Markov parameters. This issue is tightly related to the single-trajectory setup. Similar to~\cite{dean2017sample}, strict stability can be removed using multiple independent trajectories since $e_t$ is replaced by the initial state $x_0$ in each trajectory that is zero by~restarting the experiment.
This setup is briefly observed in~\cite[Remark 3]{oymak2019non}, and is used in recent work~\cite{sun2020finite} in which a standard nuclear norm regularization is discussed.  Note that only the last output measurement $y_{T-1}$ of each trajectory is used in the OLS estimation of~\cite{sun2020finite}, which facilitates the non-asymptotic analysis but might be very data inefficient. The setup of multiple trajectories was also used in~\cite{tu2017non} where the authors focused on stable SISO (single-input and single-output) systems and derived sharp bounds on $\mathcal{H}_{\infty}$ error between the true unknown plant and the estimated FIR (finite impulse response) approximation.

\textbf{Contributions:} We focus on the partially observed LTI system~\eqref{eq:dynamics} that can be open-loop stable or unstable. Our contributions are two aspects. First, motivated by~\cite{tu2017non}, we present a simple OLS to estimate the Markov parameters using multiple trajectories. In contrast to~\cite{oymak2019non,simchowitz2019learning,sarkar2019finite}, the setup of multiple trajectories allows for any spectral radius of $A$. Unlike~\cite{sun2020finite}, we utilize all data samples in each trajectory for the OLS estimator, thus improving the data efficiency. See Table~\ref{tab:comparision} for a comparison with recent work on LTI system identification. Second, we present a non-asymptotic analysis on the estimation error by leveraging standard concentration results of random matrices with special upper-triangular Toeplitz structures in data matrices. The OSL estimator is provably consistent for both stable and unstable systems, and the convergence rate scales as $\mathcal{O}(1/\sqrt{N})$ for estimating the Markov parameters, where $N$ is the number of trajectories. This rate is consistent with previous work~\cite{sun2020finite,oymak2019non,sarkar2019finite,simchowitz2019learning,tu2017non}; see Table~\ref{tab:comparision}.  Unlike fully observed LTI systems where recent work has proven that more unstable systems are easier to estimate~\cite{simchowitz2018learning,dean2017sample}, our theoretical bound suggests that partially observed LTI systems with bigger spectral radius $\rho(A)$ are harder to estimate when there exists process noise $w_t$. Interestingly, if the process noise is zero, {i.e.}, $w_t = 0$, our non-asymptotic bound is independent to $\rho(A)$. Numerical experiments on two marginally stable and unstable systems demonstrate the performance of our OLS estimator compared to that in~\cite{sun2020finite,simchowitz2019learning}.


\textbf{Organization:} The rest of this paper is organized as follows. We~present the problem statement and least-squares procedure in Section~\ref{section:statement}. Non-asymptotic analysis is discussed in Section~\ref{section:bound}, and numerical results are presented in Section~\ref{section:experiments}.~Section~\ref{section:conclusion} concludes the paper. Proofs are postponed to the appendix, where some auxiliary results on non-zero initial state, recovery of state-space models, and selecting the length of Markov parameters are also discussed.

\textbf{Notation:} Given a matrix $A \in \mathbb{R}^{m \times n}$, the Frobenius norm is denoted by $\|A\|_{\text{F}} = \sqrt{\text{Tr}(AA^\tr)}$, and we denote $\|A\|$ as its spectral norm, \emph{i.e.}, its largest singular value $\sigma_{\max}(A)$. $\rho(A)$ denotes the spectral radius of a square matrix $A$, and $A^\tr$ denotes the transpose of matrix $A$. 
For a symmetric matrix $B \in \mathbb{S}^{n}$, $\lambda_{\min}(B)$ denotes its minimum eigenvalue. Multivariate normal distribution with mean $\mu$ and covariance matrix $\Sigma$ is denoted by $\mathcal{N}(\mu,\Sigma)$.

 \begin{table*}[t]
 \addtolength{\leftskip} {-1cm}
 \addtolength{\leftskip} {-1cm}
 \setlength{\abovecaptionskip}{2mm}
   \setlength{\belowcaptionskip}{3mm}
\captionsetup{font=normal}
  \caption{A summary of recent non-asymptotic analysis for LTI system estimation}
  {\footnotesize
 \begin{tabular}{m{3.65cm} c c c c c c c c}
  \toprule
   \multirow{2}{*}{Paper}  & \multirow{2}{*}{Meas.} & \multirow{2}{*}{Type} & \multirow{2}{*}{Stability} & \multirow{2}{*}{Rollouts} & \multirow{2}{*}{Data} & \multirow{2}{*}{Inputs $u_t$} &  \multicolumn{2}{c}{Rate}  \\
    & & & & & & & $G$ & $A,B,C,D$ \\
    \hline
   This work & \multirow{7}{*}{Partial$^\dagger$} & MIMO & Any & Multiple & all & Gaussian & $\displaystyle \mathcal{O}\big({{N}^{-\frac{1}{2}}}\big)$ & \multirow{6}{*}{$\displaystyle \mathcal{O}\big({{N}^{-\frac{1}{2}}}\big)$}   \\
Sun \emph{et al.}, 2020 \cite{sun2020finite}    & & MIMO & Any  & Multiple & final  & Gaussian & $\displaystyle \mathcal{O}\big({{N}^{-\frac{1}{2}}}\big)$  \\
     Tu \emph{et al.}, 2017 \cite{tu2017non}$^{1}$&  &SISO & $\rho(A) < 1$  &  Multiple& all & Deterministic & $\displaystyle \mathcal{O}\big({{N}^{-\frac{1}{2}}}\big)$    \\
     Oymak \emph{et al.}, 2019 \cite{oymak2019non}  &  & MIMO & $\rho(A) < 1$  & Single & all & Gaussian & $\displaystyle \mathcal{O}\big({{N}^{-\frac{1}{2}}}\big)$  \\
      Sarkar \emph{et al.}, 2019 \cite{sarkar2019finite}  &  & MIMO & $\rho(A) < 1$  & Single & all& Gaussian & $\displaystyle \mathcal{O}\big({{N}^{-\frac{1}{2}}}\big)$  \\
    Simchowitz~\!et \!al.,~\!2019\cite{simchowitz2019learning}  &  & MIMO & $\rho(A) \leq 1$  & Single &all& Gaussian & $\displaystyle \mathcal{O}\big({{N}^{-\frac{1}{2}}}\big)$  \\
    \hline
    Dean \emph{et al.}, 2019 \cite{dean2017sample}  &  \multirow{3}{*}{Full$^\ddagger$} & MIMO & Any &  Multiple & final &Gaussian & - & $\displaystyle \mathcal{O}\big({{N}^{-\frac{1}{2}}}\big)$  \\
       Sarkar \emph{et al.}, 2018 \cite{sarkar2018near}  &  & MIMO &  Any$^2$  & Single &all& Gaussian &-& $\displaystyle \mathcal{O}\big({{N}^{-\frac{1}{2}}}\big)$  \\
   Simchowitz~{et al.},2018~\cite{simchowitz2018learning}  &  & MIMO &  $\rho(A) \leq 1$ &  Single & all& Gaussian & - & $\displaystyle \mathcal{O}\big({{N}^{-\frac{1}{2}}}\big)$  \\
    \bottomrule
  \end{tabular}
  }
  \raggedright
{\footnotesize
{\begin{spacing}{1}
$^\dagger$\;: For partially observed  LTI systems, the convergence rate is typically expressed in terms of estimating the Markov parameter $G$, scaling as $\displaystyle \mathcal{O}\big({{N}^{-\frac{1}{2}}}\big)$; The refined robustness result of Ho-Kalman algorithm in~\cite{sarkar2018near} (see also~\cite{tsiamis2019finite,oymak2019non}) suggests that the convergence rate for estimating $A, B, C, D$ is $\displaystyle \mathcal{O}\big({{N}^{-\frac{1}{2}}}\big)$.
For robust controller design, we might not need to estimate $A, B, C, D$ explicitly; see Appendix~\ref{appendix:Ho-kalman} for further discussion. \newline
$^\ddagger$\;: For fully observed LTI systems, the rate is expressed in terms of state space matrix estimation $A, B$; \newline
$^{1}$\;: Their method can be used for unstable SISO systems, but their results on $\mathcal{H}_\infty$ error bounds are specialized for open-loop stable SISO systems.\newline
$^2$\;: The result requires a regularity condition on eigenvalues of $A$~\cite[Section 5]{sarkar2018near}; also see~\cite[Theorem 1]{sarkar2018near} for detailed scaling of convergence.
\end{spacing}
}
}
\label{tab:comparision}
\end{table*}

\section{Problem Statement} \label{section:statement}

We consider the MIMO (multiple-input and multiple-output) LTI system~\eqref{eq:dynamics}. It is assumed that the process and measurement noises are i.i.d. Gaussian, \emph{i.e.}, $w_t \sim\mathcal{N}(0,\sigma_w^2I)$ and $v_t \sim\mathcal{N}(0,\sigma_v^2I)$. Given a horizon $T$, we aim 1) to learn the first $T$ Markov parameters of the system
\begin{equation} \label{eq:MarkovParameter}
    G  = \begin{bmatrix} D &  CB & CAB & \ldots & CA^{T-2}B \end{bmatrix} \in \mathbb{R}^{p \times mT},
\end{equation}
 and 2) to provide a finite sample bound on the estimation accuracy. Together with the classical Ho-Kalman algorithm~\cite{ho1966effective} and recent robustness results in~\cite{oymak2019non,sarkar2019finite}, a consistent estimate of $G$ can return a consistent estimation of system matrices $({A}, {B}, {C}, {D})$ up to a similarity transformation. Note that the matrices  $B_w$ and $D_v$ in~\eqref{eq:dynamics} are unknown but not estimated. 

\subsection{Data collection via multiple rollouts}
By running experiments in which the system~\eqref{eq:dynamics} starts at $x_0 = 0$\footnote{We note that $x_0$ can be random with zero mean and bounded covariance, for which our estimation procedure stays the same and the analysis can be easily carried over; see Appendix~\ref{appendix:initialstate} for details. In addition, the length of Markov parameters and the rollout length can be different; see Appendix~\ref{appendix:rollout_length}.}
and the dynamics evolve with a given input $u_t$, we record the resulting output measurements $y_t$. We call the input/output trajectory $(y_t, u_t)$ from such experiment as a \emph{rollout}. To identify the Markov parameter~\eqref{eq:MarkovParameter}, we excite the system with Gaussian noise $u_t \sim \mathcal{N}(0,\sigma_u^2 I)$ for $N$ rollouts, each of length $T$. The resulting dataset is
$$
    \left\{\left(y_t^{(i)},u_t^{(i)}\right) : 1 \leq i \leq N, \; 0 \leq t \leq T-1\right\},
$$
where $t$ indexes the time in each rollout and $i$ denotes each independent rollout.

This multi-rollout data collection procedure is motivated by~\cite{tu2017non}, where the authors focused on open-loop stable~SISO systems with optimal deterministic input design. This multi-rollout setup is also adopted in recent work~\cite{sun2020finite}, where only the last output measurement $y_{T-1}^{(i)}$ is recorded in each rollout. A similar multi-rollout setup is used  for estimating fully observed LTI systems in~\cite{dean2017sample}. Another widely-used data collection procedure is based on a \emph{single rollout}~\cite{simchowitz2018learning,simchowitz2019learning,oymak2019non,sarkar2019finite,sarkar2018near}, where we apply an input sequence from $0$ to $N \times T -1$, without restart, and collect all the outputs. 
Table~\ref{tab:comparision} lists a summary of recent non-asymptotic analysis of LTI system identification in different setups.

\subsection{Least-squares procedure}

To ease notation, for each rollout $i$, we organize the input/output data as
\begin{equation}
    \begin{aligned}
    y^{(i)} &= \begin{bmatrix}y^{(i)}_0 & y^{(i)}_1 & \ldots & y^{(i)}_{T-1} \end{bmatrix} \in \mathbb{R}^{p \times T}, \\
        u^{(i)} &= \begin{bmatrix}u^{(i)}_0 & u^{(i)}_1 & \ldots & u^{(i)}_{T-1} \end{bmatrix} \in \mathbb{R}^{m \times T},
\end{aligned}
\end{equation}
and the process/measurement noise  as
\begin{equation}
    \begin{aligned}
    w^{(i)} &= \begin{bmatrix}w^{(i)}_0 & w^{(i)}_1 & \ldots & w^{(i)}_{T-1} \end{bmatrix} \in \mathbb{R}^{q \times T}, \\
        v^{(i)} &= \begin{bmatrix}v^{(i)}_0 & v^{(i)}_1 & \ldots & v^{(i)}_{T-1} \end{bmatrix} \in \mathbb{R}^{l \times T}.
\end{aligned}
\end{equation}
To establish an explicit connection with Markov parameters, each measurement $y_t, t = 0, \ldots, T-1$ can be expanded recursively as (where we used the fact $x_0^{(i)} = 0$)
\begin{equation} \label{eq:output_t}
    \begin{aligned}
        y^{(i)}_{t} &= Cx^{(i)}_{t} + Du^{(i)}_{t} + D_v v_{t} \\
        &= \sum_{k=0}^{t-1}CA^k\left(Bu^{(i)}_{t-k-1} + B_w w^{(i)}_{t-k-1}\right) + Du^{(i)}_{t} + D_v v^{(i)}_{t}. \\
    \end{aligned}
\end{equation}
Upon defining two upper-triangular Toeplitz matrices corresponding to $u^{(i)}$ and  $w^{(i)}$, respectively,
\begin{equation} \label{eq:U&W}
\begin{aligned}
    U^{(i)} &= \begin{bmatrix}
        u_0^{(i)} & u_1^{(i)} & u_2^{(i)} & \ldots & u^{(i)}_{T-1} \\
        0 & u_0^{(i)} & u_1^{(i)} & \ldots & u^{(i)}_{T-2}  \\
         0 & 0 & u_0^{(i)} & \ldots & u^{(i)}_{T-3} \\
        \vdots & \vdots & \vdots & \ddots & \vdots \\
         0 & 0 & 0 & \ldots & u^{(i)}_{0} \\
        \end{bmatrix}  \in \mathbb{R}^{mT \times T}, \\
        W^{(i)} &= \begin{bmatrix}
        w_0^{(i)} & w_1^{(i)} & w_2^{(i)} & \ldots & w^{(i)}_{T-1} \\
        0 & w_0^{(i)} & w_1^{(i)} & \ldots & w^{(i)}_{T-2}  \\
         0 & 0 & w_0^{(i)} & \ldots & w^{(i)}_{T-3} \\
        \vdots & \vdots & \vdots & \ddots & \vdots \\
         0 & 0 & 0 & \ldots & w^{(i)}_{0}  \\
        \end{bmatrix} \in \mathbb{R}^{qT \times T},
\end{aligned}
\end{equation}
and $ F = \begin{bmatrix} 0 &  CB_w & CAB_w & \ldots & CA^{T-2}B_w \end{bmatrix} \in \mathbb{R}^{p \times qT}$,
the measurement data~\eqref{eq:output_t} in each rollout $i$ can be compactly written as
\begin{equation} \label{eq:outputi}
    y^{(i)} =  GU^{(i)} + FW^{(i)} + D_vv^{(i)}.
\end{equation}
%
%
Note that the noise terms $W^{(i)}$ and $v^{(i)}$ are zero-mean. We form the OLS problem as
\begin{equation} \label{eq:leastsquares}
    \hat{G} = \arg\min_{X \in \mathbb{R}^{p \times mT}} \sum_{i=1}^N \|y^{(i)}- X U^{(i)}\|^2_{\text{F}}.
\end{equation}
Defining our label matrix $Y$ and input data matrix $U$ as
$$
\begin{aligned}
 Y &= \begin{bmatrix} y^{(1)} & \ldots& y^{(N)} \end{bmatrix} \in \mathbb{R}^{p \times NT },\\
 U &= \begin{bmatrix} U^{(1)}& \ldots& U^{(N)}  \end{bmatrix} \in \mathbb{R}^{mT \times NT},
\end{aligned}
$$
the OLS becomes
$$
    \min_{X \in \mathbb{R}^{p \times mT}} \quad \|Y - XU\|_{\text{F}}^2.
$$
Hence, the least square solution is $\hat{G} = YU^\dagger$
where $U^\dagger = U^\tr(UU^\tr)^{-1}$ is the right pseudo-inverse of $U$. Ideally, we would like the estimation error $\|G - \hat{G}\|$ to be small. Our main result bounds this error as a function of sample size $N$, length $T$, and noise levels $\sigma_w, \sigma_v$, and system parameters.
\vspace{-1mm}
\begin{remark}[Open-loop stability] \label{remark:stability}
For any fixed $T$, the noise terms $FW^{(i)}$, $D_vv^{(i)}$ in~\eqref{eq:outputi} are zero mean with finite covariance, irrespective of $\rho(A) < 1$ or $\rho(A) \geq 1$.
As shown  in Section~\ref{section:bound}, the OLS~\eqref{eq:leastsquares} is suitable for estimating both open-loop stable and unstable systems.
For the single-rollout setup in~\cite{oymak2019non,simchowitz2019learning,sarkar2019finite}, we need to generate $N$ subsequences of length $T$, each of which evolves as
$
    y_t = G \bar{u}_t + F\bar{w}_t + D_vv_t + e_t,
$
with $\bar{u}_t = \begin{bmatrix} u_t^\tr & u_{t-1}^\tr & \ldots & u_{t - T + 1}^\tr  \end{bmatrix}^\tr \in \mathbb{R}^{mT}, \bar{w}_t = \begin{bmatrix} w_t^\tr & w_{t-1}^\tr & \ldots & w_{t - T + 1}^\tr  \end{bmatrix}^\tr \in \mathbb{R}^{qT},$
where the extra term $e_t = CA^{T-1}x_{t-T+1}$ is due to the state at time $t-T+1$. The requirement of bounding this term makes the methods of~\cite{oymak2019non,sarkar2019finite} only suitable for strictly stable systems. 
Unfortunately, as pointed out in~\cite{simchowitz2019learning}, the condition $\rho(A) <1$ is quite restrictive, because simple oscillators and integrator from Newton's law yield systems with $\rho(A) = 1$. For example, the matrix $
 A = \begin{bmatrix} 1 & \Delta \\ 0 & 1 \end{bmatrix}
 $,
 corresponding to a discretization of Newton's second law, violates the strict stability. More recently, the authors in~\cite{simchowitz2019learning} have extended~\cite{oymak2019non} for marginally stable systems $\rho(A) \leq 1$ using a prefiltered OLS; see Table~\ref{tab:comparision} for their comparison.
\end{remark}

\section{Results on Learning Markov Parameters} \label{section:bound}

In this section, we work out the statistical rate for the OLS estimator~\eqref{eq:leastsquares} which uses all data samples of each trajectory. 
Since $W^{(i)}$ and $U^{(i)}$ are independent, and $v^{(i)}$ and $U^{(i)}$ are also independent, our analysis is simpler than that in~\cite{oymak2019non}. Similar to~\cite{dean2017sample}, our analysis ideas leverage  standard non-asymptotic analysis of random matrices with exploitation of the special Toeplitz structures of $W^{(i)}$ and $U^{(i)}$ in~\eqref{eq:U&W}.

\subsection{Theoretical bound on least-squares error} \label{subsection:mainresult}
From~\eqref{eq:outputi}, the systems of outputs are
$$
    Y = GU + FW + D_vv,
$$
where $W =  \begin{bmatrix} W^{(1)}, \; \ldots, \; W^{(N)} \end{bmatrix} \in \mathbb{R}^{qT \times NT },
v = \begin{bmatrix} v^{(1)}, \; \ldots, \; v^{(N)} \end{bmatrix} \in \mathbb{R}^{l \times NT}.$
%
Following the OLS~\eqref{eq:leastsquares}, the estimation error is given by
\begin{equation} \label{eq:error}
\begin{aligned}
    \hat{G} - G &= YU^\tr(UU^\tr)^{-1} - G  \\
                & = (Y - GU)U^\tr(UU^\tr)^{-1} \\
                & = (FW+D_vv)U^\tr(UU^\tr)^{-1}.
\end{aligned}
\end{equation}
Our main theorem quantifies the estimation error that captures the problem dependencies in term of noise levels $\sigma_w$, $\sigma_v$, system parameters {$F$, $D_v$}, and the dimension $m+q+l$.
\begin{theorem} \label{theo:mainresult}
 Fix any $0<\delta <1$. If the number of rollouts satisfies $N \geq 8mT + 4(m+q+l+4)\log(3T/\delta)$, we have with probability at least $1 - \delta$
    \begin{equation} \label{eq:MainTheorem}
        \|\hat{G} - G\| \leq \frac{\sigma_v C_1 + \sigma_wC_2}{\sigma_u}\sqrt{\frac{1}{N}},
    \end{equation}
    where
    $$
    \begin{aligned}
       C_1 &=  8\|D_v\|\sqrt{2T(T+1)(m+l)\log(27T/\delta)},  \\
       C_2 &= 16\|F\|\sqrt{\left(\frac{T^3}{3} + \frac{T^2}{2} + \frac{T}{6}\right)2(m+q)\log(27T/\delta)}.
    \end{aligned}
    \vspace{1mm}
    $$
\end{theorem}

Although this theorem is not hard to prove, we have a few interesting implications. First, the bound~\eqref{eq:MainTheorem} states that the estimation error $\|\hat{G} - G\|$ behaviors as $\mathcal{O}(\frac{1}{\sqrt{N}})$, which is consistent with previous non-asymptotic results (they may scale differently with respect to other system quantities); see Table~\ref{tab:comparision}.
Second, our bound individually accounts for the process noise $w_t$ and measurement noise $v_t$, while the results in~\cite{sun2020finite} assume no process noise $w_t = 0$. In addition, we have some further observations below.  

\begin{itemize}[align=left] 
\item \textit{Stable vs. unstable systems.} Theorem~\ref{theo:mainresult} confirms that the OLS~\eqref{eq:leastsquares} returns a consistent estimation of $G$ for both stable and unstable systems, while results in~\cite{oymak2019non,sarkar2019finite,simchowitz2019learning} only work for stable systems due to their single-rollout setup, as discussed in Remark~\ref{remark:stability}. Recent results show that when states are directly observed, unstable systems with bigger $\rho(A)$ are easier the estimate in both multi-rollout~\cite{dean2017sample} and single-rollout~\cite{simchowitz2019learning} setups. However, this statement might not be true for partially observed LTI systems. In particular, the constant $\|F\|$ will grow exponentially when $\rho(A) >1$ in Theorem~\ref{theo:mainresult}, suggesting that more trajectories are needed for unstable systems to achieve the same accuracy. This is expected since the process noise $w_t$ will be amplified during the system evolution for unstable systems. Interestingly, when the process noise $w_t = 0$, the bound in Theorem~\ref{theo:mainresult} is independent to  $\rho(A)$. 

    \item \textit{Dependency of system dimensions.} The bound~\eqref{eq:MainTheorem} is not tight with respect to system dimensions. First, there are $pmT$ parameters in $G$ to learn, and our bound states that we need $\mathcal{O}(mT)$ trajectories (which is also required in~\cite{sun2020finite}), each of which provides $pT$ measurements. There is a redundancy of factor $T$. Second, the constants $C_1$ and $C_2$ in Theorem~\ref{theo:mainresult} depends polynomially on the length $T$ of $G$, suggesting more trajectories are needed to estimate longer Markov Parameters. However, our numerical results (see Fig.~\ref{fig:unstable_varyingT}) suggest that the normalized estimation error seems unrelated to the length $T$ after utilizing all data samples in~\eqref{eq:leastsquares}. It is interesting to derive tighter bounds in future work.
    To estimate $A, B, C, D$ using the Ho-Kalman algorithm~\cite{ho1966effective}, the value of $T$ should be big enough to satisfy a rank condition on the Hankel matrix; see Appendix~\ref{appendix:lengthT} for further discussion on selecting $T$. 
\end{itemize}



Finally, using the celebrated Ho-Kalman algorithm~\cite{ho1966effective} and recent robustness analysis in\cite{oymak2019non,sarkar2018near,tsiamis2019finite}, we can obtain consistent state-space representation $\hat{A}, \hat{B}, \hat{C}, \hat{D}$ up to a similarity transformation. We summarize this observation in the following corollary; 
See Appendix~\ref{appendix:Ho-kalman} for details on the Ho-Kalman algorithm.

\begin{corollary}[Recovery of System Parameters]  \label{corollary:ABCD}
Suppose that $(A, B, C, D)$ in~\eqref{eq:dynamics} is controllable and observable. Then, for $N$ and $T$ sufficiently large, there exists a unitary matrix $S$, and a constant $C_3$ depending on system parameters $(A, B, C, D)$, the dimension $(n, m, p, q, l, T)$ and $\sigma_u, \sigma_w, \sigma_v$, and logarithmic factor of $N$ such that we have with high probability
    \begin{equation*} 
        \max\{ \|\hat{A} - SAS^\tr\|, \|\hat{B} - SB\|, \|\hat{C} - CS^\tr\|, \|\hat{D} - D\| \} \leq \frac{C_3 }{N^\frac{1}{2}},
    \end{equation*}
    where $(\hat{A}, \hat{B}, \hat{C}, \hat{D})$ is the output of the Ho-Kalman algorithm on the OLS estimation $\hat{G}$ from~\eqref{eq:leastsquares}.
\end{corollary}

\subsection{Proof sketch of Theorem~\ref{theo:mainresult}} \label{section:proofsketch}

The basic proof idea is standard and similar to~\cite{dean2017sample,oymak2019non,sun2020finite}. The spectral norm of the estimation error~\eqref{eq:error} can be bounded as
\begin{equation*}
    \|\hat{G} - G\| \leq \left\|(UU^\tr)^{-1}\right\|\left(\|F\|\left\|WU^\tr\right\| + \|D_v\|\left\|vU^\tr\right\|\right).
\end{equation*}
Each  term of the right-hand side will be bounded individually. In particular, we can show that with high probability
$
        \left\|(UU^\tr)^{-1}\right\|$ behaves as  $\mathcal{O}\left(\frac{1}{N}\right)$,
     $   \left\|WU^\tr\right\|$ behaves as $ \mathcal{O}(\sqrt{N})$, and
     $   \left\|vU^\tr\right\|$ behaves as $ \mathcal{O}(\sqrt{N})$.
Combining these terms leads to~\eqref{eq:MainTheorem}.
Our proof relies on two standard lemmas in concentration analysis of random matrices. The first lemma bounds the spectral norm of the product of two Gaussian matrices~\cite[Lemma 1]{dean2017sample}.
\begin{lemma} \label{Lemma:spectralnorm}
    Fix a $\delta \in (0,1)$ and $N \geq 2(m + n) \log(1/\delta)$. Let $f_k \in \mathbb{R}^m, g_k \in \mathbb{R}^n$ be independent random vectors $f_k \sim \mathcal{N}(0, \sigma_f^2I_n), g_k \sim \mathcal{N}(0, \sigma_g^2I_m), k = 1, \ldots N$. With probability at least $1 - \delta$,
    $$
        \left\|\sum_{k=1}^N f_k g_k^\tr\right\| \leq 4\sigma_f\sigma_g \sqrt{N(m + n)\log(9/\delta)}.
    \vspace{2mm}
    $$
\end{lemma}

The second lemma gives a non-asymptotic bound on the minimum singular value of a standard Wishart matrix~\cite[Lemma 2]{dean2017sample}.
\begin{lemma} \label{lemma:singularvalue}
    Let $u_k \sim \mathcal{N}(0, \sigma_u^2 I_m), k = 1, \ldots, N$ be i.i.d vectors. With probability at least $1 - \delta$, we have
    $$
        \sqrt{\frac{1}{\sigma_u^2}\lambda_{\min}\left( \sum_{k=1}^N u_ku_k^\tr\right)} \geq \sqrt{N} - \sqrt{m} - \sqrt{2\log(1/\delta)}.
        \vspace{2mm}
    $$
\end{lemma}

Using these two lemmas, we show bounds of $\|UU^\tr\|$, $\left\|vU^\tr\right\|$, $\left\|WU^\tr\right\|$ in the following propositions, and their proofs are provided in Appendix~\ref{appendix:proofs}.
\begin{proposition} \label{prop:U}
Fix a $0<\delta <1$, and $N \geq 8mT + 16 \log(T/\delta)$. We have with probability at least $1 - \delta$
\begin{equation} \label{eq:UUmin}
         \lambda_{\min}(UU^\tr) \geq \frac{1}{4}\sigma_u^2 N.
\vspace{2mm}
\end{equation}
\end{proposition}

\begin{proposition} \label{prop:vU}
Fix a  $0<\delta <1$, and $N \geq 2(m + l) \log(T/\delta)$. We have with probability at least $1 - \delta$
    $$
          \|vU^\tr\| \leq 2\sigma_v\sigma_u\sqrt{2T(T+1)N(m+l)\log(9T/\delta)}.
          \vspace{2mm}
    $$
\end{proposition}

\begin{proposition} \label{prop:WU}
Fix a $0<\delta <1$, and $N \geq 4(m + q) \log(T/\delta)$. We have with probability at least $1 - \delta$
    $$
          \|WU^\tr\| \leq 4\sigma_w\sigma_u\sqrt{\left(\frac{T^3}{3}\! + \!\frac{T^2}{2}\! +\! \frac{T}{6}\right)2N(m+q)\log(9T/\delta)}.
    \vspace{2mm}
    $$
\end{proposition}

With the results in Propositions~\ref{prop:U}---\ref{prop:WU}, Theorem~\ref{theo:mainresult} can be easily proved: Let $N \geq 8mT + 4(m+q+l+4)\log(3T/\delta)$. Combining the estimates in Propositions~\ref{prop:U}---\ref{prop:WU} via union bound, we have with probability at least $1-\delta$,
$$
    \begin{aligned}
         &\|\hat{G} - G\| \\
         \leq &\left\|(UU^\tr)^{-1}\right\|\left(\|F\|\left\|WU^\tr\right\| + \|D_v\|\left\|vU^\tr\right\|\right) \\
         \leq &\frac{4}{\sigma_u N} \Bigg(2\sigma_v\|D_v\|\sqrt{2T(T+1)N(m+l)\log(27T/\delta)}  \\
         & \qquad \qquad + \left. 4\sigma_w\|F\|\sqrt{\left(\frac{T^3}{3} + \frac{T^2}{2} + \frac{T}{6}\right)2N(m+q)\log(27T/\delta)}\right) \\
         = &\frac{\sigma_v C_1+ \sigma_w C_2}{\sigma_u} \sqrt{\frac{1}{N}},
    \end{aligned}
$$
where the parameters $C_1, C_2$ are defined in Theorem~\ref{theo:mainresult}.

We conclude this section by reiterating that details on Corollary~\ref{corollary:ABCD} are postponed to Appendix~\ref{appendix:Ho-kalman}, and that discussions of non-zero initial state and the length $T$ of Markov parameter $G$ are provided in Appendices~\ref{appendix:initialstate} and~\ref{appendix:lengthT}, respectively.






\section{Numerical Experiments} \label{section:experiments}

We present a series of numerical experiments to demonstrate the performance of the OLS estimator~\eqref{eq:leastsquares}. We compare our method with the estimation approaches in~\cite{sun2020finite,simchowitz2019learning} for marginally stable systems, and with~\cite{sun2020finite} for unstable systems. Note that other methods in~\cite{oymak2019non,sarkar2019finite} are only suitable for strictly stable systems. Recall that~\cite{sun2020finite} used the same multi-rollout setup but only records the last data sample $y_{T-1}^{(i)}$, and~\cite{simchowitz2019learning} introduced a pre-filtered OLS relying on a single-rollout setup for which we set the length as $NT$. All experiments were carried out in MATLAB, and our scripts can be downloaded from {\small \url{https://github.com/zhengy09/SysId}}.

\subsection{Estimation of marginally stable and unstable systems}

We first consider a marginally stable system
\begin{equation} \label{eq:marginallystable}
\begin{aligned}
    A &= \begin{bmatrix} 1 & \Delta \\ 0 & 1 \end{bmatrix}, \;  B = \begin{bmatrix} 0\\1 \end{bmatrix}, \; C = \begin{bmatrix} 1 & 0 \end{bmatrix}, \\
    D &= 0, \; B_w = B, \; D_v = 1,
\end{aligned}
\end{equation}
with $\Delta = 0.2$, which corresponds to a discretization of Netwon's second law. We also consider an open-loop unstable system, adapted from~\cite{dean2017sample}, as follows
\begin{equation} \label{eq:unstable}
\begin{aligned}
    A &= \begin{bmatrix} 1.01 & 0.01 & 0 \\ 0.01 & 1.01 & 0.01 \\ 0 & 0.01 & 1.01 \end{bmatrix}, \;   \; C = \begin{bmatrix} 1 & 0 & 0 \end{bmatrix}, \\
   B &= I_3,\; D = 0, \; B_w = I_3, \; D_v = 1.
\end{aligned}
\end{equation}
In our data collection, we used inputs with $\sigma_u = 1$, and noises with $\sigma_w = 0.2$ and $\sigma_v = 0.5$. For the multi-rollout setup, the rollout length is set as $T = 10$, and we vary the number of rollouts from 50 to 500.  For the method in~\cite{simchowitz2019learning}, we run a new simulation with a single rollout of length $NT$.
The behavior of different OLS estimates is illustrated in Fig.~\ref{fig:stable&unstable_systems}. As expected, increasing the number of rollouts reduces the estimation errors for all methods. For all scenarios, utilizing all data samples in our method returns a significantly better estimation than using the final data sample only in~\cite{sun2020finite}. Note that we chose a typical set of parameters for the pre-filtered OLS~\cite{simchowitz2019learning}, and by finely tuning the parameters, its performance might be better than reported here.

\begin{figure}[t]
    \centering
     \hspace{-2mm}
	\subfigure[]
	{
    \includegraphics[scale = 0.53]{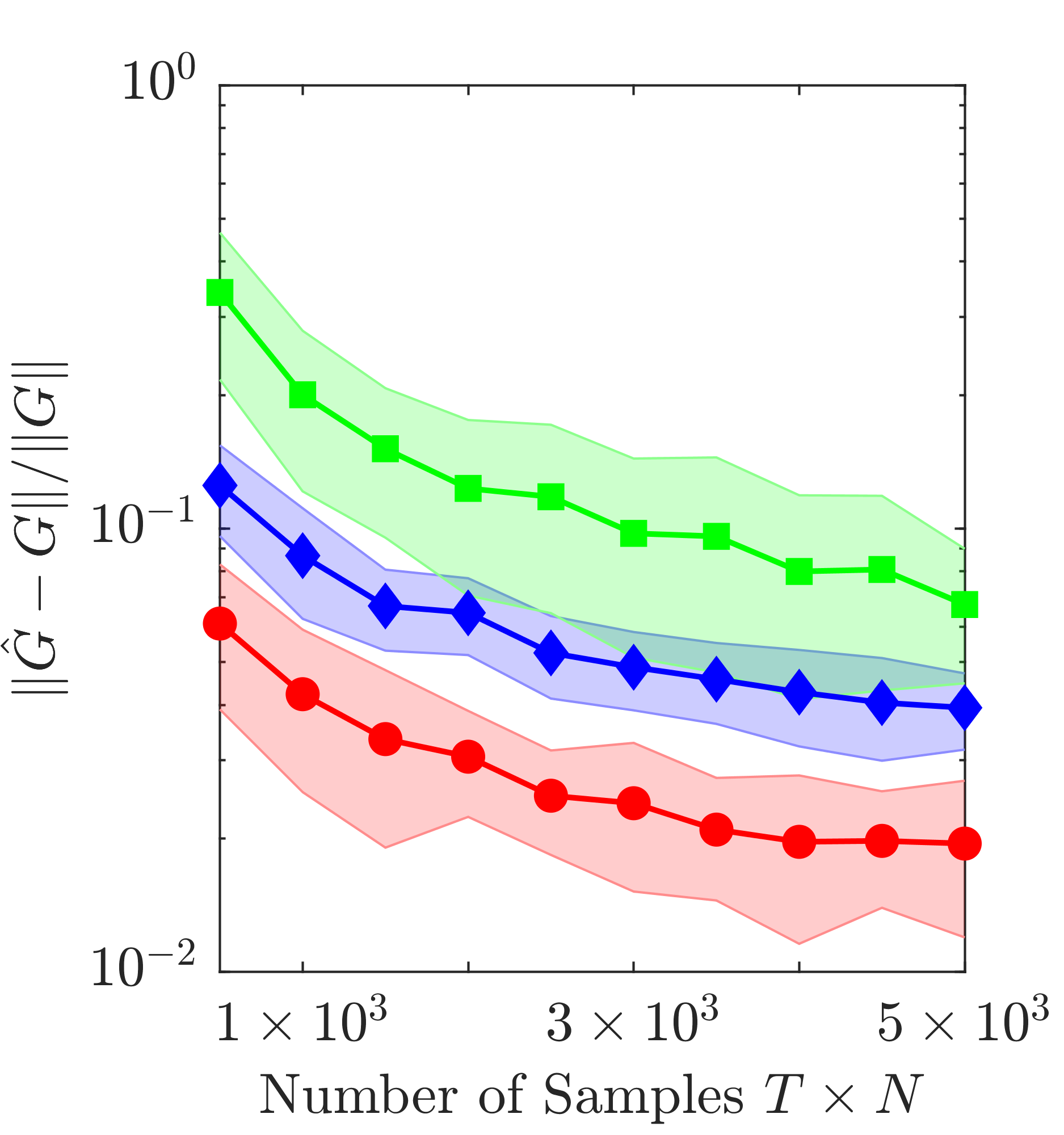}
    }
        \hspace{10mm}
    \subfigure[]
	{
    \includegraphics[scale = 0.53]{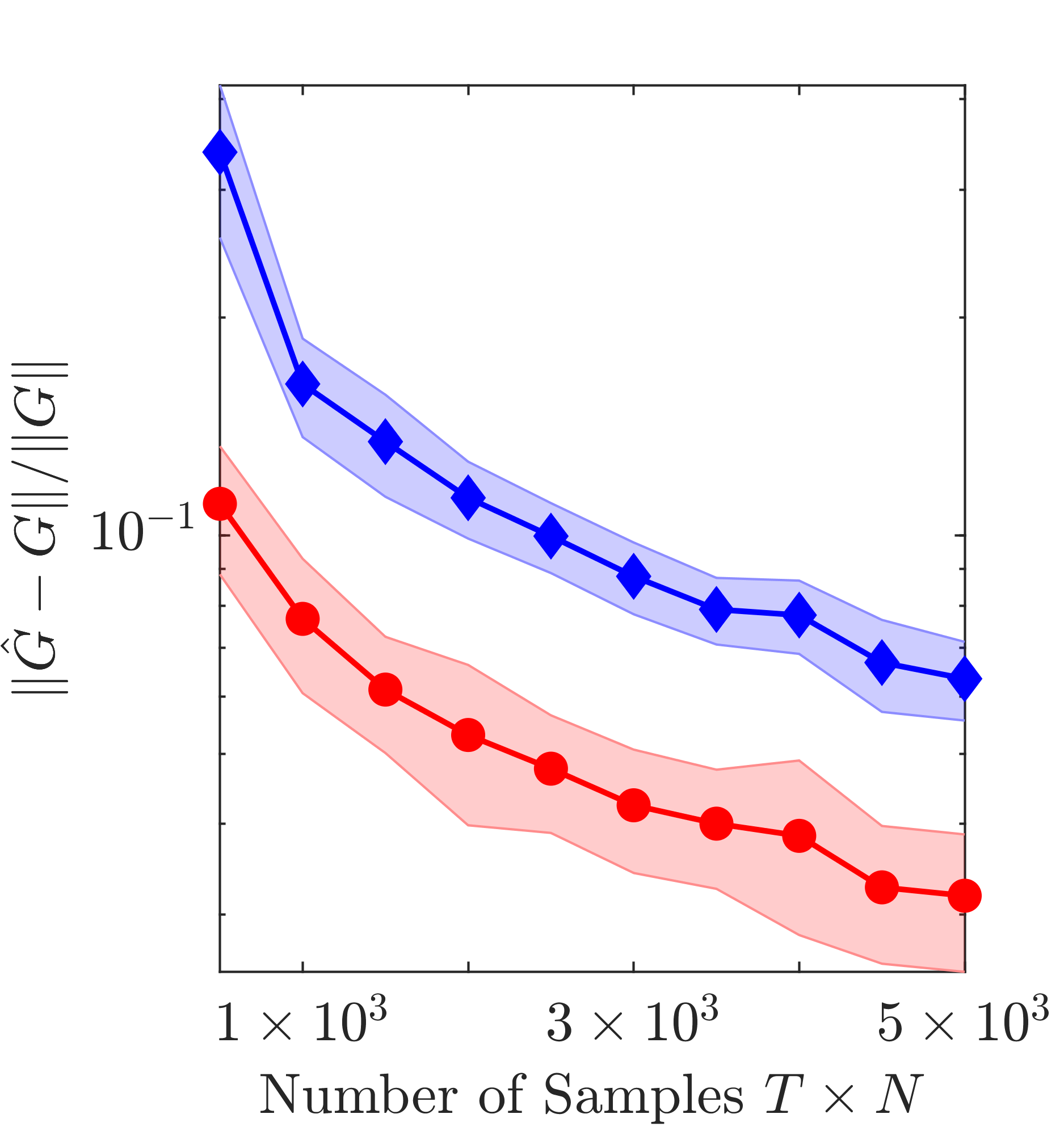}
    }
    \caption{Comparison with different methods: (a) marginally stable system $\lambda_1 = \lambda_2 = 1$~\eqref{eq:marginallystable}; (b) open-loop unstable system~\eqref{eq:unstable}. The green curve with square dots denotes the OLS from~\cite{simchowitz2019learning} (single rollout); The blue curve with diamond dots denotes the OLS from~\cite{sun2020finite} (multi-rollout); The red curve with circle dots denotes the OLS~\eqref{eq:leastsquares} in our work.}
    \label{fig:stable&unstable_systems}
    \label{fig:stable&unstable_systems}
\end{figure}

\begin{figure*}[t]
\centerfloat
\setlength{\abovecaptionskip}{2pt}
  	{
    \includegraphics[scale=0.7]{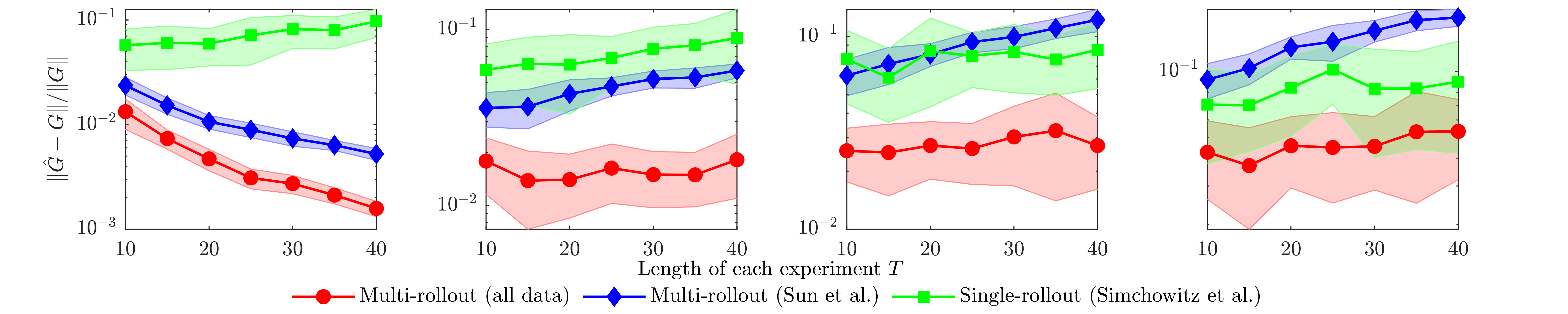}
    }
    \caption{Estimation error for system~\eqref{eq:marginallystable} with varying length $T$ and process noises. Left to right: $\sigma_w = 0$; $\sigma_w = 0.2$; $\sigma_w = 0.4$; $\sigma_w = 0.6$}
    \label{fig:stable_varyingT}
 \end{figure*}

\begin{figure*}[t]
    \centerfloat
    \setlength{\abovecaptionskip}{2pt}
    \includegraphics[scale=0.7]{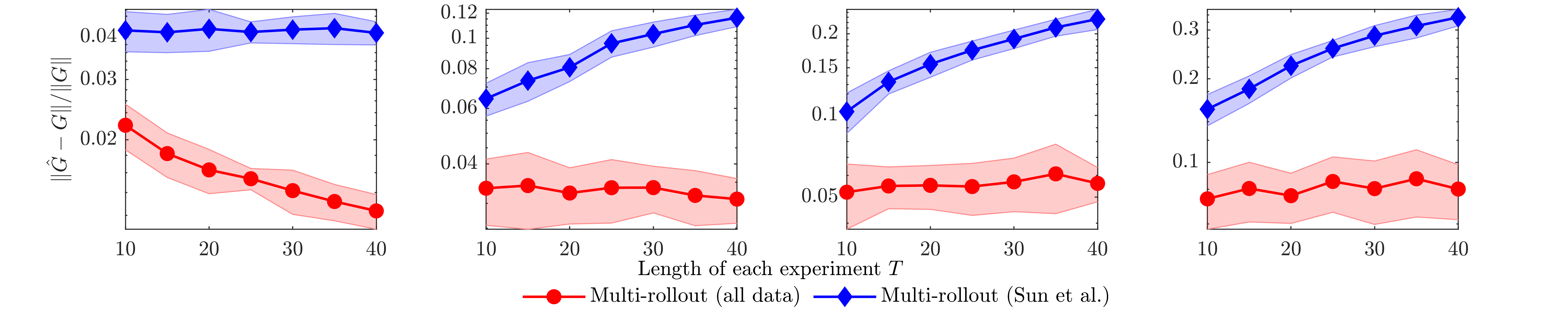}
    \caption{Estimation error for system~\eqref{eq:unstable} with varying length $T$ and process noises. Left to right: $\sigma_w = 0$; $\sigma_w = 0.2$; $\sigma_w = 0.4$; $\sigma_w = 0.6$.}
    \label{fig:unstable_varyingT}
    \vspace{-2mm}
\end{figure*}

\subsection{Experiments with varying length $T$ of Markov Parameters}

Here, we present numerical results in which we fix the number of rollouts $N = 500$ and vary the rollout length from 10 to 40. The standard deviations of input and measurement noise are $\sigma_u =1, \sigma_v = 0.5$, and we vary the standard deviation of process noise from 0 to 0.6. The normalized estimation errors are illustrated in Fig.~\ref{fig:stable_varyingT} and Fig.~\ref{fig:unstable_varyingT}. Interestingly, varying the length $T$ of the Markov parameters seems to have little impact on the normalized errors of the OLS~\eqref{eq:leastsquares} and the prefiltered OLS in~\cite{simchowitz2019learning}, while the performance of~\cite{sun2020finite} worsens quickly when the process noise $\sigma_w$ is non-zero. Again, this confirms the benefits of utilizing all the data samples, and also indicate the theoretical bound in Theorem~\ref{theo:mainresult} is not tight. Another interesting point is that the behavior of the prefiltered OLS in~\cite{simchowitz2019learning} seems independent of the process noise, suggesting that prefiltering can effectively mitigate the noise level. It would be interesting to investigate how to incorporate prefiltering in the multi-rollout setup.

\subsection{Experiments with varying spectral radius $\rho(A)$}

Our final experiment demonstrates the influence of the spectral radius $\rho(A)$ on estimation performance. We consider a system with two inputs, two outputs, and state dimension being three, \emph{i.e.}, $n = 3, m = 2, p = 2$. We generate the system data as follows: matrix $A$ with random integers from $1$ to $5$, and matrices $B,C,D$ with random integers from $-2$ to $2$.   Then, we re-scale the matrix $A$ to adjust its spectral radius $\rho(A)$. In our simulations, we tested $\rho(A)$ from $1$ to $10$. The numerical results are shown in Figure~\ref{fig:varyingA}. In this experiment, we fixed the length of Markov parameters as $T = 9$, and the number of rollouts as $N = 1000$. We used random inputs with $\sigma_u = 1$, and assumed measurement noises with $\sigma_v = 0.5$.
As clearly observed in Figure~\ref{fig:varyingA}, when there is no process noise $\sigma_w = 0$ (the left subfigure), the estimation error is independent to the spectral radius $\rho(A)$. When the process noise becomes non-zero (the middle and right subfigures), the estimation errors become worse quickly as $\rho(A)$ goes bigger. These results are consistent with the theoretical prediction in Theorem~\ref{theo:mainresult}.

\begin{figure}[h]
        \centering
        \includegraphics[scale=0.6]{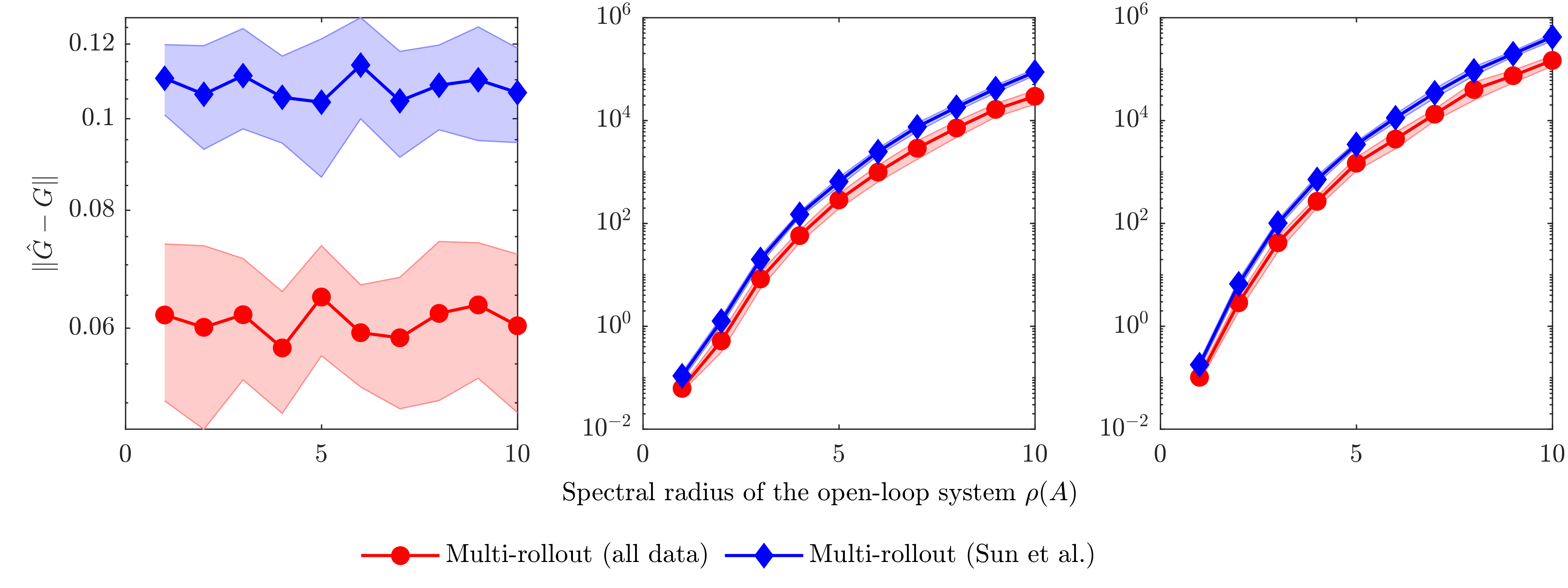}
        \caption{Estimation error with varying spectral radius $\rho(A)$ and process noises. Left to right: $\sigma_w = 0$; $\sigma_w = 0.01$; $\sigma_w = 0.05$.}
        \label{fig:varyingA}
    \end{figure}

\section{Conclusions} \label{section:conclusion}

We have introduced a simple OLS estimator for partially observed LTI system identification using multiple~trajectories and analyzed its corresponding finite sample complexity. Our method utilizes all data samples in the trajectories, thus making very efficient use of the available data, as~demonstrated in our experiments. For future work, it would be interesting to improve our analysis for the OLS estimator, and to derive certain minimax risk lower bounds for identifying general LTI systems in multi-rollout setup. It is also interesting to investigate other algorithmic variants (e.g., using outputs in a prefiltering step~\cite{simchowitz2019learning,tsiamis2019finite}) to mitigate system instability, and to adapt
analysis techniques for asymptotic normality~\cite{chiuso2004asymptotic} and data-dependent finite sample analysis~\cite{campi2005guaranteed} in the multi-rollout setup. Similar to the Coarse-ID procedure~\cite{dean2017sample}, we are also interested in combining the results with robust synthesis techniques (e.g.,~\cite{zhou1996robust,zheng2019equivalence,furieri2019input,anderson2019system}) to derive end-to-end guarantees for data-driven control.



\section*{Acknowledgement}

We would like to thank Yujie Tang and Runyu Zhang for insightful discussions on this work. YZ would also like to thank Luca Furieri for nice conversations that have lead to this work.


\appendix
\vspace{10mm}
\noindent\textbf{\Large Appendix}
\section{Proofs of Propositions~\ref{prop:U}-\ref{prop:WU}} \label{appendix:proofs}
In this appendix, we provide the proofs of Propositions~\ref{prop:U}-\ref{prop:WU}. For self-completeness, we first introduce three technical lemmas.

 \begin{lemma}[{\cite{horn2012matrix}}] \label{lemma:PSDeigen}
     Given a positive semidefinite matrix $A \in \mathbb{S}^n_+$, and another symmetric matrix $B \in \mathbb{S}^n$, we have
     $$
        \lambda_{\min}(A + B) \geq \lambda_{\min}(B).
        \vspace{2mm}
     $$
 \end{lemma}

\begin{lemma}[{\cite{horn2012matrix}}]\label{lemma:frobeniusnorm}
 For any matrix  $A \in \mathbb{R}^{m \times n}$, we have
 $$
    \|A\| \leq \|A\|_{\text{F}} \leq \sqrt{\min\{m,n\}} \|A\|.
    \vspace{2mm}
 $$
 \end{lemma}
 \begin{lemma}[{\cite[Lemma 2.10]{zhou1996robust}}] \label{lemma:blockmatrix}
     Let $A$ be a block partitioned matrix with
     $$
        A = \begin{bmatrix} A_{11} & A_{12} & \ldots &A_{1n} \\
        A_{21} & A_{22} & \ldots & A_{2n}\\
        \vdots & \vdots & \ddots & \vdots \\
        A_{m1} & A_{m2} & \ldots & A_{mn}\end{bmatrix},
     $$
     where each $A_{ij}$ has compatible dimensions. Then, we have
     $$
        \|A\| \leq \left\|\begin{bmatrix} \|A_{11}\| & \|A_{12}\| & \ldots &\|A_{1n}\| \\
        \|A_{21}\| & \|A_{22}\| & \ldots & \|A_{2n}\|\\
        \vdots & \vdots & \ddots & \vdots \\
        \|A_{m1}\| & \|A_{m2}\| & \ldots & \|A_{mn}\|\end{bmatrix}\right\|.
     \vspace{2mm}
     $$
 \end{lemma}
 This lemma is true for any induced norm; see~\cite[Lemma 2.10]{zhou1996robust} for a proof.




\subsection{Proof of Proposition~\ref{prop:U}}

   Upon denoting the input data matrix of each rollout $i$ as $
        U^{(i)} = \begin{bmatrix} \hat{u}^{(i)}_0 & \ldots & \hat{u}^{(i)}_{T-1} \end{bmatrix} \in \mathbb{R}^{mT\times T},
   $    where
   $$
        \hat{u}^{(i)}_t = \begin{bmatrix} \left(u^{(i)}_t\right)^\tr,\, \ldots,\,  \left(u^{(i)}_0\right)^\tr,\, 0,\, \ldots,\, 0 \end{bmatrix}^\tr \in \mathbb{R}^{mT},
   $$
   for $t = 0, \ldots, T-1$,  we can write
   $$
   \begin{aligned}
     UU^\tr &= \sum_{i=1}^N U^{(i)}\left(U^{(i)}\right)^\tr \\
     &= \sum_{i=1}^N \sum_{t=0}^{T-1} \hat{u}^{(i)}_t(\hat{u}^{(i)}_t)^\tr  \\
     &= \sum_{t=0}^{T-1}  \sum_{i=1}^N \hat{u}^{(i)}_t(\hat{u}^{(i)}_t)^\tr. \\
   \end{aligned}
   $$
   For each $t$, the nonzero part of $\hat{u}^{(i)}_t$ follows a norm distribution with zero mean and covariance $\sigma_u^2I_{m(t+1)}$. According to Lemma~\ref{lemma:singularvalue}, we have with probability at least $1-\delta/T$
   $$
   \begin{aligned}
    &\sqrt{\frac{1}{\sigma_u^2}\lambda_{\min}\left( E_{t}\sum_{i=1}^N u_t^{(i)}\left(u_t^{(i)}\right)^\tr E_t^\tr \right)} \\
    \geq &\sqrt{N} - \sqrt{m(t+1)} - \sqrt{2\log(T/\delta)}, \quad t = 0, \ldots, T-1,
    \end{aligned}
    $$
    where $E_t$ is a basis matrix that extracts the nonzero block.

    Combining all of the estimates above via union bound, when $\sqrt{N} > \sqrt{mT} + \sqrt{2\log(T/\delta)}$,   by Lemma~\ref{lemma:PSDeigen}, we have with probability at least $1-\delta$,
    $$
     UU^\tr = \sum_{t=0}^{T-1}  \sum_{i=1}^N \hat{u}^{(i)}_t(\hat{u}^{(i)}_t)^\tr \succeq \sum_{i=1}^N \hat{u}^{(i)}_{T-1}(\hat{u}^{(i)}_{T-1})^\tr, \\
    $$
    which implies that
    $$
    \begin{aligned}
      \sqrt{\frac{1}{\sigma_u^2}\lambda_{\min} \left(UU^\tr\right)} &\geq  \sqrt{\frac{1}{\sigma_u^2}\lambda_{\min}\left(\sum_{i=1}^N u_{T-1}^{(i)}\left(u_{T-1}^{(i)}\right)^\tr \right)}\\
      &\geq \sqrt{N} - \sqrt{mT} - \sqrt{2\log(T/\delta)}.
     \end{aligned}
    $$
    The desired result in~\eqref{eq:UUmin} follows by choosing the value of $N$ such that
    $$
        \frac{1}{2}\sqrt{N} \geq \sqrt{mT} + \sqrt{2\log(T/\delta)}.
    $$
    This holds when $N \geq 8mT + 16 \log(T/\delta)$, considering the inequality $(a + b)^2 \leq 2(a^2 + b^2).$

\subsection{Proof of Proposition~\ref{prop:vU}}
For each rollout $i$, we can write
$$
\begin{aligned}
    &v^{(i)}(U^{(i)})^\tr \\
    = &\begin{bmatrix} v_0^{(i)} & v_1^{(i)} & \ldots & v_{T-1}^{(i)} \end{bmatrix}\begin{bmatrix}
        u_0^{(i)} & u_1^{(i)} & u_2^{(i)} & \ldots & u^{(i)}_{T-1} \\
        0 & u_0^{(i)} & u_1^{(i)} & \ldots & u^{(i)}_{T-2}  \\
         0 & 0 & u_0^{(i)} & \ldots & u^{(i)}_{T-3} \\
        \vdots & \vdots & \vdots & \ddots & \vdots \\
         0 & 0 & 0 & \ldots & u^{(i)}_{0} \\
        \end{bmatrix}^\tr  \\
        = &\begin{bmatrix} \displaystyle \sum_{k=0}^{T-1}v_k^{(i)}(u_k^{(i)})^\tr,\,  \displaystyle\sum_{k=0}^{T-2}v_{k+1}^{(i)}(u_k^{(i)})^\tr,\,  \ldots,\,  \displaystyle v^{(i)}_{T-1}(u^{(i)}_0)^\tr  \end{bmatrix}.
\end{aligned}
$$
Therefore, we have
\begin{equation*} 
\begin{aligned}
vU^\tr
= &\sum_{i=1}^{N} v^{(i)}(U^{(i)})^\tr  \\
     = &\displaystyle \sum_{i=1}^{N} \begin{bmatrix} \displaystyle\sum_{k=0}^{T-1}v_{k}^{(i)}(u_k^{(i)})^\tr,\,  \displaystyle\sum_{k=0}^{T-2}v_{k+1}^{(i)}(u_k^{(i)})^\tr,\, \ldots,\,  \displaystyle v^{(i)}_{T-1}(u^{(i)}_0)^\tr  \end{bmatrix}.
\end{aligned}
\end{equation*}
Applying Lemma~\ref{Lemma:spectralnorm} to each individual block,  when $N(T-j) \geq 2(m + l) \log(T/\delta)$, we have with probability at least $1 - \delta/T$,
    $$
    \begin{aligned}
             & \left\|\sum_{i=1}^{N}\sum_{k=0}^{T-j-1}v_{k+j}^{(i)}(u_k^{(i)})^\tr \right\|
        \leq4 \sigma_v\sigma_u c \sqrt{(T-j)}
    \end{aligned}
    $$
    with $c = \sqrt{N(m + l)\log(9T/\delta)}$ for $j =0 ,\ldots T-1$. Finally, union bounding over all these events,  by Lemma~\ref{lemma:blockmatrix},  if $N\geq 2(m + l) \log(T/\delta)$, we have with probability at least $1 - \delta$ the following inequality holds: 
    \begin{equation*} 
       \begin{aligned}
        &\,\left\|\sum_{i=1}^{N} v^{(i)}(U^{(i)})^\tr\right\| \\
        \leq &\, \left\|\begin{bmatrix} \left\|\displaystyle \sum_{i=1}^{N}\sum_{k=0}^{T-1}v_{k}^i(u_k^i)^\tr\right\|  &  \left\|\displaystyle\sum_{i=1}^{N}\sum_{k=0}^{T-2}v_{k+1}^i(u_k^i)^\tr\right\|  &  \ldots &  \displaystyle \left\|\sum_{i=1}^{N}v^i_{T-1}(u^i_0)^\tr\right\| \end{bmatrix}\right\| \\
        \leq & 4\sigma_v\sigma_u\left\|\begin{bmatrix} \sqrt{NT(m + l)\log(9T/\delta)},\, \sqrt{N(T-1)(m + l)\log(9T/\delta)},\,  \ldots,\,  \sqrt{N(m + l)\log(9T/\delta)} \end{bmatrix}\right\|\\
        = & \, 4 \sigma_v\sigma_u\sqrt{N(m + l)\log(9/\delta)}\left\|\begin{bmatrix}  \sqrt{T}  &  \sqrt{T-1}  &  \ldots &  1 \end{bmatrix}\right\|   \\
        = & \, 2\sigma_v\sigma_u\sqrt{2T(T+1)N(m+l)\log(9T/\delta)}.
    \end{aligned}
    \end{equation*}
    This completes the proof.

\subsection{Proof of Proposition~\ref{prop:WU}}
Here, we prove the bound for
$$
    \|WU^\tr\| = \left\|\sum_{i=1}^N W^{(i)}(U^{(i)})^\tr\right\|.
$$
We first give bounds on each $q \times m$ block of the data matrix $WU^\tr$ that can be expressed as
$$
    (WU^\tr)_{lj} = \sum_{i=1}^N \sum_{t=0}^{T-j} w^{(i)}_{t}(u^{(i)}_{t+l-j})^\tr, \, 1 \leq j \leq T,  1\leq l \leq T.
$$
Applying Lemma~\ref{Lemma:spectralnorm} to the submatrix above, if $N(T-j+1) > 4(q+m)\log(T/\delta) $, we have with probability at least $1 - \delta/T^2$
\begin{equation*}
      \left\|(WU^\tr)_{lj}\right\| \leq 4\sigma_w\sigma_u\sqrt{2N(T-j+1)(m+q)\log(9T/\delta)},
\end{equation*}
for $1\leq l \leq T, 1 \leq j \leq T$. We define a scaling index
$$
    \begin{aligned}
   M:=&\left\|\begin{bmatrix}  \sqrt{T} &\sqrt{T-1} & \sqrt{T-2} & \ldots & 1 \\
            * & \sqrt{T-1} & \sqrt{T-2}&\ldots & 1 \\
             * & * & \sqrt{T-2}& \ldots & 1 \\
                \vdots & \vdots & \vdots & \ddots & \vdots \\
                   * & * & *& \ldots &1   \end{bmatrix}\right\|
\end{aligned}
$$
where $*$ denotes the symmetric part. By Lemma~\ref{lemma:frobeniusnorm}, we have
$$
    \begin{aligned}
   M \leq \left\|\begin{bmatrix}  \sqrt{T} &\sqrt{T-1} & \sqrt{T-2} & \ldots & 1 \\
            * & \sqrt{T-1} & \sqrt{T-2}&\ldots & 1 \\
             * & * & \sqrt{T-2}& \ldots & 1 \\
                \vdots & \vdots & \vdots & \ddots & \vdots \\
                   * & * & *& \ldots &1   \end{bmatrix}\right\|_{\text{F}}               =&  \sqrt{2 \sum_{k=1}^T \left(\sum_{l=1}^{k} l\right) - (1 + 2 + \ldots + T)} \\
                  =& \sqrt{\sum_{k=1}^T (k^2+k) -  (1 + 2 + \ldots + T)} \\
                  =& \sqrt{\frac{T^3}{3} + \frac{T^2}{2} + \frac{T}{6}}.
\end{aligned}
$$

Finally, union bounding all the estimates of individual blocks, if $N > 4(q+m)\log(T/\delta)$, according to Lemma~\ref{lemma:blockmatrix}, we have with probability $1 - \delta$
$$
\begin{aligned}
    \|WU^\tr\| \leq & M \times 4\sigma_w\sigma_u\sqrt{2N(m+q)\log(9T/\delta)} \\
         \leq & \sqrt{\frac{T^3}{3} + \frac{T^2}{2} + \frac{T}{6}} 4\sigma_w\sigma_u\sqrt{2N(m+q)\log(9T/\delta)},
\end{aligned}
$$
which completes the proof.

\section{Non-zero initial state in each rollout} \label{appendix:initialstate}

In this section, we discuss the influence of non-zero initial state in each rollout on the OLS estimator. We highlight a minor modification of the theoretical bound in Theorem~\ref{theo:mainresult} when the initial state follows an iid Gaussian distribution.

Suppose that the initial state $x_0^{(i)}$ of the $i$-th rollout satisfies $x_0^{(i)} \sim \mathcal{N}(0, \sigma_{0}^2I_n)$. Each measurement $y_t, t = 0, \ldots, T-1$ in~\eqref{eq:output_t} becomes
\begin{equation} \label{eq:output_x0_t}
    \begin{aligned}
        y^{(i)}_{t} &= Cx^{(i)}_{t} + Du^{(i)}_{t} + D_v v_{t} \\
        &= \underbrace{\sum_{k=0}^{t-1}CA^k\left(Bu^{(i)}_{t-k-1} + B_w w^{(i)}_{t-k-1}\right)  + Du^{(i)}_{t}}_{\text{input \& process noise}} + \underbrace{CA^{t}x_0^{(i)}}_{\text{init. state}} + \underbrace{D_v v^{(i)}_{t}}_{\text{meas. noise}}, \\
    \end{aligned}
\end{equation}
which is affected by input and process noises, initial state, measurement noise. Upon defining
$$
    H = \begin{bmatrix} C & CA & CA^2& \cdots & CA^{T-1} \end{bmatrix} \in \mathbb{R}^{p \times nT},
$$
the measurement data~\eqref{eq:output_x0_t} in each rollout $i$ can be compactly written as
\begin{equation} \label{eq:output_y_x0}
 y^{(i)} = GU^{(i)} + FW^{(i)} + H \left(I_T \otimes x_0^{(i)}\right) + D_vv^{(i)}.
\end{equation}
Since the state information is not directly observable when $C \neq I$, we treat the initial state $x_0$ as noise, similar to the process/measurement noises $w_t, v_t$.


In this case, we can formulate the same OLS estimator~\eqref{eq:leastsquares}. Now, we define
$$
    X_0 = \begin{bmatrix} I_T\otimes x_0^{(1)} & \cdots & I_T\otimes x_0^{(N)} \end{bmatrix} \in \mathbb{R}^{nT \times NT}.
$$
Then, the estimation error in~\eqref{eq:error} becomes
\begin{equation} \label{eq:error_x0}
\begin{aligned}
    \hat{G} - G &= YU^\tr(UU^\tr)^{-1} - G  \\
                & = (Y - GU)U^\tr(UU^\tr)^{-1} \\
                & = (FW+HX_0 + D_vv)U^\tr(UU^\tr)^{-1}.
\end{aligned}
\end{equation}
Similar to Section~\ref{section:proofsketch}, we have
\begin{equation*}
    \|\hat{G} - G\| \leq \left\|(UU^\tr)^{-1}\right\|\left(\|F\|\left\|WU^\tr\right\| + \|H\| \|X_0U^{\tr}\|+\|D_v\|\left\|vU^\tr\right\|\right).
\end{equation*}
We only need to bound the error term introduced by the non-zero initial state, given by the following proposition.
\begin{proposition} \label{prop:XU}
Fix a $0<\delta <1$, and $N \geq 4(n+m) \log(T/\delta)$. We have with probability at least $1 - \delta$
    $$
          \|X_0U^\tr\| \leq 4\sigma_0\sigma_u\sqrt{T(T+1)N(m+n)\log(9T/\delta)}.
    \vspace{2mm}
    $$
\end{proposition}
\begin{proof}
   The proof is similar to that of Proposition~\ref{prop:vU}/\ref{prop:WU}. Consider that
   $$
   \begin{aligned}
    \|X_0U^\tr\| &= \left\|\sum_{i=1}^N \left(I_T\otimes x_0^{(i)}\right)(U^{(i)})^\tr\right\| \\
    &=\left\|\sum_{i=1}^N \begin{bmatrix}
         x_0^{(i)}(u_0^{(i)})^\tr & 0 & 0 & \ldots & 0 \\
         x_0^{(i)}(u_1^{(i)})^\tr &  x_0^{(i)}(u_0^{(i)})^\tr & 0 & \ldots & 0  \\
         x_0^{(i)}(u_2^{(i)})^\tr &  x_0^{(i)}(u_1^{(i)})^\tr &  x_0^{(i)}(u_0^{(i)})^\tr & \ldots & 0 \\
        \vdots & \vdots & \vdots & \ddots & \vdots \\
          x_0^{(i)}(u^{(i)}_{T-1})^\tr &  x_0^{(i)}(u^{(i)}_{T-2})^\tr &  x_0^{(i)}(u^{(i)}_{T-3})^\tr & \ldots &  x_0^{(i)}(u^{(i)}_{0})^\tr \\
        \end{bmatrix}\right\|.
    \end{aligned}
$$
Applying Lemma~\ref{Lemma:spectralnorm} to each block of the matrix above, if $N > 4(n+m)\log(T/\delta) $, we have with probability at least $1 - \delta/T^2$
\begin{equation*}
      \left\|(X_0U^\tr)_{lj}\right\| \leq 4\sigma_0\sigma_u\sqrt{2N(m+n)\log(9T/\delta)}, \qquad 1\leq l \leq T, 1 \leq j \leq l.
\end{equation*}
Union bounding all the estimates of individual blocks, if $N > 4(n+m)\log(T/\delta)$, according to Lemma~\ref{lemma:blockmatrix}, we have with probability $1 - \delta$
$$
\begin{aligned}
    \|X_0U^\tr\| \leq & \left\| \begin{bmatrix} 1 & 0 & 0 & \ldots & 0 \\
    1 & 1 & 0 & \ldots & 0\\
    1 & 1 & 1 & \ldots & 0 \\
    \vdots & \vdots & \vdots & \ddots & \vdots \\
    1 & 1 & 1 &\ldots & 1\end{bmatrix} \right\| 4\sigma_0\sigma_u\sqrt{2N(n+m)\log(9T/\delta)} \\
   \leq & \left\| \begin{bmatrix} 1 & 0 & 0 & \ldots & 0 \\
    1 & 1 & 0 & \ldots & 0\\
    1 & 1 & 1 & \ldots & 0 \\
    \vdots & \vdots & \vdots & \ddots & \vdots \\
    1 & 1 & 1 &\ldots & 1\end{bmatrix} \right\|_{\text{F}} 4\sigma_0\sigma_u\sqrt{2N(n+m)\log(9T/\delta)} \\
         = & 4\sigma_0\sigma_u\sqrt{T(T+1)N(m+q)\log(9T/\delta)}.
\end{aligned}
$$
We complete the proof.
\end{proof}
By slightly adjusting the proof of Theorem~\ref{theo:mainresult}, we have the following corollary.
\begin{corollary} \label{corollary:initialstate}
 Suppose the initial state $x_0 \sim \mathcal{N}(0, \sigma_0^2 I_n)$, process/measurement noises $w_t \sim \mathcal{N}(0, \sigma_w^2 I_q), v_t \sim \mathcal{N}(0, \sigma_v^2 I_l)$, and we choose the input $u_t \sim \mathcal{N}(0, \sigma_u^2 I_m)$. Fix any $0<\delta <1$. If the number of rollouts satisfies $N \geq 8mT + 4(m+n+q+l+4)\log(4T/\delta)$, we have with probability at least $1 - \delta$
    \begin{equation} \label{eq:MainTheorem}
        \|\hat{G} - G\| \leq \frac{\sigma_0 C_0 + \sigma_v C_1 + \sigma_wC_2 }{\sigma_u}\sqrt{\frac{1}{N}},
    \end{equation}
    where
    $$
    \begin{aligned}
        C_0 &=  16\|H\|\sqrt{T(T+1)(m+n)\log(36T/\delta)},  \\
       C_1 &=  8\|D_v\|\sqrt{2T(T+1)(m+l)\log(36T/\delta)},  \\
       C_2 &= 16\|F\|\sqrt{\left(\frac{T^3}{3} + \frac{T^2}{2} + \frac{T}{6}\right)2(m+q)\log(36T/\delta)}. \\
    \end{aligned}
    \vspace{1mm}
    $$
\end{corollary}

\section{Recovery of system parameters and Corollary~\ref{corollary:ABCD}} \label{appendix:Ho-kalman}

The celebrated Ho-Kalman algorithm~\cite{ho1966effective} generates a controllable and observable state-space realization $(A, B, C,D)$ from the Markov parameter matrix $G$ in~\eqref{eq:MarkovParameter}. Recall that an LTI system $(A, B, C,D)$ is called \emph{controllable} if
$$
    \text{rank}\left(\begin{bmatrix} B & AB & \ldots & A^{n-1}B \end{bmatrix}\right) = n,
$$
and \emph{observable} if
$$
    \text{rank}\left(\begin{bmatrix} C \\ CA \\ \vdots \\ CA^{n-1}\end{bmatrix}\right) = n.
$$
Note that state-space realizations from the Markov parameters $G$ are not unique, but they are only different up to a similarity transformation.

To apply the Ho-Kalman algorithm, we need to construct the Hankel matrix $\mathcal{H} \in \mathbb{R}^{pT_1 \times m(T_2+1)}$ from the Markov parameter $G$
$$
    \mathcal{H} = \begin{bmatrix} CB & CAB & CA^2B & \ldots & CA^{T_2}B \\
                        CAB & CA^2B & CA^3B & \ldots & CA^{T_2+1}B \\
                        CA^2B & CA^3B & CA^4B & \ldots & CA^{T_2 + 2}B \\
                        \vdots & \vdots & \vdots & \vdots & \vdots \\
                        CA^{T_1-1}B & CA^{T_1} & CA^{T_1 + 1} & \ldots & CA^{T_1 + T_2-1}B\end{bmatrix},
$$
where $T = T_1 + T_2 + 1$. We also denote $\mathcal{H}^{-}$ to be the $pT_1 \times mT_2$ Hankel matrix created by dropping the last block column of $\mathcal{H}$. It is easy to check that
$$
    \mathcal{H} = \begin{bmatrix} C \\ CA \\ \vdots \\ CA^{T_1-1} \end{bmatrix}\begin{bmatrix} B & AB & \ldots & A^{T_2}B \end{bmatrix}, \qquad     \mathcal{H}^{-} = \begin{bmatrix} C \\ CA \\ \vdots \\ CA^{T_1-1} \end{bmatrix}\begin{bmatrix} B & AB & \ldots & A^{T_2-1}B \end{bmatrix}.
$$
Then, under the assumption of the controllability and observability of $(A,B,C,D)$, both the Hankel matrix $\mathcal{H}$ and its modified version $\mathcal{H}^{-}$ are of rank $n$, as long as $\min(T_1, T_2) \geq n$.

\begin{center}
    \begin{tcolorbox}[width=0.8\textwidth,
                  boxsep=0pt,
                  colback=white
                  ]
The Ho-Kalman algorithm~\cite{ho1966effective} on the noiseless Markov parameter $G$ has the following steps:
\begin{enumerate}[leftmargin=5mm]
\setlength{\itemsep}{0ex}
    \item Perform the singular value decomposition (SVD): $\mathcal{H}^{-} = U\Sigma V$, where $\Sigma \in \mathbb{R}^{n \times n}$.
    \item Compute the observability matrix $\mathcal{O} = U\Sigma^{\frac{1}{2}}$, and return the first $p \times n$ submatrix of $\mathcal{O}$ as $C$;
    \item Compute the controllability matrix $\mathcal{C} = \Sigma^{\frac{1}{2}}V$, and return the  first $n \times m$ submatrix of $\mathcal{C}$ as $B$;
    \item Return $A=\mathcal{O}^\dagger \mathcal{H}^{+} \mathcal{C}^\dagger$, where $ \mathcal{H}^{+} $ is the $pT_1 \times mT_2$ Hankel matrix by dropping the first block column of~$\mathcal{H}$.
\end{enumerate}
Note that $D$ is the first $p \times m$ submatrix of ${G}$.
\end{tcolorbox}
\end{center}

We can also apply the Ho-Kalman algorithm on the estimated Markov parameter $\hat{G}$; see~\cite[Algorithm 1]{oymak2019non}. The work~\cite{oymak2019non} shows how the state-space estimations $\hat{A}, \hat{B}, \hat{C}, \hat{D}$ degrade when using a noisy estimation $\hat{G}$. Precisely, we denote $\hat{\mathcal{H}}$ as the analogous estimated version of $\mathcal{H}$. The following result is taken from~\cite[Corollary 5.4]{oymak2019non}.
\begin{proposition}[{\cite[Corollary 5.4]{oymak2019non}}] \label{corollary:hokalman}
    Suppose that $(A, B, C, D)$ in~\eqref{eq:dynamics} is controllable and observable. Let $(\hat{A}, \hat{B}, \hat{C}, \hat{D})$ is the output of the Ho-Kalman algorithm on the OLS estimation $\hat{G}$. Suppose
    $$
        \|\mathcal{H} - \hat{\mathcal{H}}\| \leq \frac{\sigma_{\min} \left(\mathcal{H}^{-}\right)}{4}.
    $$
    Then, there exists an unitary matrix $S$ such that
    $$
    \begin{aligned}
        \max\{ \|\hat{B} - SB\|, \|\hat{C} - CS^\tr \|\} \leq 5 \sqrt{n \|\mathcal{H} - \hat{\mathcal{H}}\|}, \qquad
        \|\hat{A} - SAS^\tr\| \leq \frac{50\sqrt{n\|\mathcal{H} - \hat{\mathcal{H}}\|} \|\mathcal{H}\|}{\sigma_{\min}^{3/2}\left(\mathcal{H}^{-}\right)}.
    \end{aligned}
    $$
\end{proposition}
As stated in~\cite[Lemma 5.2]{oymak2019non}, we have
\begin{equation} \label{eq:HG}
    \|\mathcal{H} - \hat{\mathcal{H}}\| \leq \sqrt{\min\{T_1, T_2 + 1\}}\|G - \hat{G}\|.
\end{equation}
Since $\|\hat{D} - D\| \leq \|G - \hat{G}\|$,  for $N$ sufficiently large and $T$ large enough to have $\min\{T_1, T_2\} \geq n$, we can combine \eqref{eq:HG} and Proposition~\ref{corollary:hokalman} with Theorem~\ref{theo:mainresult} to arrive at the result in Corollary~\ref{corollary:ABCD} with a scaling of $\mathcal{O}(N^{-1/4})$. A recent refined analysis of Ho-Kalman algorithm in~\cite{sarkar2018near} enables a rate of estimation
$$
    \max\{  \|\hat{A} - SAS^\tr\|, \|\hat{B} - SB\|, \|\hat{C} - CS^\tr \|\} \leq \hat{C} \|G- \hat{G}\|,
$$
where $\hat{C}$ is a problem-dependent constant (see \cite[Section 5.3]{sarkar2018near} for details). Combining this result with Theorem~\ref{theo:mainresult}, we arrive at the result in Corollary~\ref{corollary:ABCD} with a scaling of $\mathcal{O}(N^{-1/2})$.

\section{Selecting the length of Markov Parameters $T$} \label{appendix:lengthT}

As shown in Theorem~\ref{theo:mainresult} and the discussions therein, the constants $C_1$ and $C_2$ depends on polynomials on the length of Markov parameter $T$, suggesting a longer Markov parameter $G$ requires more trajectories to achieve the same accuracy. In our multi-rollout setup, we note that $T$ is also the length of each experiment. Therefore, we might want to use a small $T$ to reduce the workload of running experiments and improve sample efficiency. However, to guarantee the rank condition on the Hankel matrices $\mathcal{H}, \mathcal{H}^{-}$ in the Ho-Kalman algorithm, we need to make sure that $\min\{T_1, T_2\} \geq n$, where $T_1 + T_2 +1 = T$. In principle, the value of $T$ should at least be $2n + 1$, where $n$ is the state dimension.

On the other hand, we can represent open-loop stable systems in the frequency domain as follows
\begin{equation} \label{eq:plantFIR}
 \mathbf{G}(z) = C(zI - A)^{-1}B + D = \sum_{i=0}^\infty G(k)\frac{1}{z^k},
\end{equation}
where $G(k) \in \mathbb{R}^{p \times m}$ denotes the $k$th spectral component. The form~\eqref{eq:plantFIR} is known as infinite impulse responses of $\mathbf{G}$. In particular, we have
$$
    G(k) = \begin{cases}
        D & k = 0, \\
        C A^{k-1} B& k \geq 1.
        \end{cases}
$$
Upon denoting the least-squares solution $\hat{G} = \begin{bmatrix} \hat{G}_0 & \hat{G}_1 & \ldots & \hat{G}_{T-1} \end{bmatrix}$, we form the estimated FIR dynamics
\begin{equation*} 
    \hat{\mathbf{G}} := \sum_{k=0}^{T-1} \hat{G}_k \frac{1}{z^k}.
\end{equation*}
If we want to bound the $\mathcal{H}_\infty$ norm of $\mathbf{G} - \hat{\mathbf{G}}$, we can choose a value of $T$ to balance the truncation error and OLS estimation error
\begin{equation*} 
\begin{aligned}
    \|\mathbf{G} - \hat{\mathbf{G}}\|_{\mathcal{H}_\infty} &= \left\| \sum_{t=0}^{T-1} \left(G(k) - \hat{G}_k\right) \frac{1}{z^k} + \sum_{k=T}^{\infty} G(k) \frac{1}{z^k}  \right\|_{\mathcal{H}_\infty} \\
    & \leq  \underbrace{\left\| \sum_{t=0}^{T-1} \left(G(k) - \hat{G}_k\right) \frac{1}{z^k}\right\|_{\mathcal{H}_\infty}}_{\text{OLS estimation error}} \; + \;\; \underbrace{\left\| \sum_{k=T}^{\infty} G(k) \frac{1}{z^k}  \right\|_{\mathcal{H}_\infty}}_{\text{FIR truncation error}}.
\end{aligned}
\end{equation*}
This strategy above has been widely used in robust controller synthesis after system identification; see e.g.,~\cite{tu2017non,zhou1996robust,simchowitz2020improper}, where an explicit state-space realization is not required.


\section{The length of Markov parameters and the rollout length can be different} \label{appendix:rollout_length}

In OLS estimator~\eqref{eq:leastsquares}, the rollout length of each experiment and the length of Markov Parameters are the same. It is possible to use different lengths in the multi-rollout setup. Suppose we aim to learn $T_1$ Markov parameters and the rollout length is $T_2$. Naturally, we need to ensure that $T_2 \geq T_1$.

Now for each rollout $i$, we organize all the data as
\begin{equation}
    \begin{aligned}
    y^{(i)} &= \begin{bmatrix}y^{(i)}_0 & y^{(i)}_1 & \ldots & y^{(i)}_{T_2-1} \end{bmatrix} \in \mathbb{R}^{p \times T_2}, \\
    v^{(i)} &= \begin{bmatrix}v^{(i)}_0 & v^{(i)}_1 & \ldots & v^{(i)}_{T_2-1} \end{bmatrix} \in \mathbb{R}^{l \times T_2}.
\end{aligned}
\end{equation}
Each measurement $y_t, t = 0, \ldots, T_2 - 1$ can be expanded recursively as (where we used the fact $x_0^{(i)} = 0$)
\begin{equation} \label{eq:output_t_v2}
    \begin{aligned}
        y^{(i)}_{t} &= Cx^{(i)}_{t} + Du^{(i)}_{t} + D_v v_{t} \\
        &= \sum_{k=0}^{t-1}CA^k\left(Bu^{(i)}_{t-k-1} + B_w w^{(i)}_{t-k-1}\right) + Du^{(i)}_{t} + D_v v^{(i)}_{t}. \\
    \end{aligned}
\end{equation}
Upon defining two upper-triangular Toeplitz matrices corresponding to $u^{(i)}$ and  $w^{(i)}$, respectively,
\begin{equation} \label{eq:U&W_v2}
\begin{aligned}
    \hat{U}^{(i)} &= \begin{bmatrix}
        u_0^{(i)} & u_1^{(i)} & u_2^{(i)} & \ldots & u^{(i)}_{T_1-1} & \ldots & u^{(i)}_{T_2-1}\\
        0 & u_0^{(i)} & u_1^{(i)} & \ldots & u^{(i)}_{T_1-2}  & \ldots & u^{(i)}_{T_2-2}\\
         0 & 0 & u_0^{(i)} & \ldots & u^{(i)}_{T_1-3} & \ldots & u^{(i)}_{T_2-3}\\
        \vdots & \vdots & \vdots & \ddots & \vdots & \ddots & \vdots\\
         0 & 0 & 0 & \ldots & u^{(i)}_{0} & \ldots &  u^{(i)}_{T_2 - T_1}\\
        \end{bmatrix}  \in \mathbb{R}^{mT_1 \times T_2}, \\
        \hat{W}^{(i)} &= \begin{bmatrix}
        w_0^{(i)} & w_1^{(i)} & w_2^{(i)} & \ldots & w^{(i)}_{T-1}  & \ldots & w^{(i)}_{T_2-1}\\
        0 & w_0^{(i)} & w_1^{(i)} & \ldots & w^{(i)}_{T-2} & \ldots & w^{(i)}_{T_2-2} \\
         0 & 0 & w_0^{(i)} & \ldots & w^{(i)}_{T-3} & \ldots & w^{(i)}_{T_2-3}\\
        \vdots & \vdots & \vdots & \ddots & \vdots & \ddots & \vdots\\
         0 & 0 & 0 & \ldots & w^{(i)}_{0} & \ldots & w^{(i)}_{T_2-T_1} \\
        \end{bmatrix} \in \mathbb{R}^{qT_1 \times T_2},
\end{aligned}
\end{equation}
and
$$
\begin{aligned}
F = \begin{bmatrix} 0 &  CB_w & CAB_w & \ldots & CA^{T_1-2}B_w \end{bmatrix} \in \mathbb{R}^{p \times qT_1},
\end{aligned}
$$
and a residual signal matrix
$$
\begin{aligned}
E^{(i)} = CA^{T_1-1} \begin{bmatrix} 0 & \ldots & 0 &  x_1^{(i)} & x_2^{(i)} & \ldots & x_{T_2 -T_1 + 1}^{(i)} \end{bmatrix} \in \mathbb{R}^{p \times T_2},
\end{aligned}
$$
the measurement data~\eqref{eq:output_t_v2} in each rollout $i$ can be compactly written as
\begin{equation} \label{eq:outputi_v2}
    y^{(i)} =  G\hat{U}^{(i)} + \underbrace{F\hat{W}^{(i)}}_{\text{proc. noise}} + \underbrace{D_vv^{(i)}}_{\text{meas. noie}} + \underbrace{E^{(i)}}_{\text{init. noie}},
\end{equation}
where the output is affected by three types of noises: process noise, measurement noise, and the noise introduced by the non-zero state $x_{t - T_1 + 1}$. Note that the noise terms $W^{(i)}$, $v^{(i)}$ and $E^{(i)}$ are zero-mean. We form the OLS problem as
\begin{equation} \label{eq:leastsquares_v2}
    \hat{G} = \arg\min_{X \in \mathbb{R}^{p \times mT_1}} \sum_{i=1}^N \|y^{(i)}- X \hat{U}^{(i)}\|^2_{\text{F}}.
\end{equation}
Defining our label matrix $Y$ and input data matrix $U$ as
$$
\begin{aligned}
 Y &= \begin{bmatrix} y^{(1)} & \ldots& y^{(N)} \end{bmatrix} \in \mathbb{R}^{p \times NT_2 },\\
 \hat{U} &= \begin{bmatrix} \hat{U}^{(1)}& \ldots& \hat{U}^{(N)}  \end{bmatrix} \in \mathbb{R}^{mT_1 \times NT_2},
\end{aligned}
$$
the OLS becomes
$$
    \min_{X \in \mathbb{R}^{p \times mT_1}} \quad \|Y - X\hat{U}\|_{\text{F}}^2,
$$
and the least square solution is $\hat{G} = Y\hat{U}^\dagger$. Following the analysis in Theorem~\ref{theo:mainresult}, it is not difficult to see that the OLS estimator~\eqref{eq:leastsquares_v2} returns a consistent estimation.

\textit{Comparing the OLS estimator~\eqref{eq:leastsquares} and the OLS estimator~\eqref{eq:leastsquares_v2}:} It is easy to see that they are equivalent when $T_1 = T_2$. When fixing the length of Markov parameter $T_1$ and increasing the rollout length $T_2$, we have two methods to estimate the $T_1$  Markov parameters:
$$
    G_{T_1} = \begin{bmatrix} D & CB & CAB& \ldots & CA^{T_1 - 1}B \end{bmatrix} \in \mathbb{R}^{p \times mT_1}
$$
The first method is to use the OLS estimator~\eqref{eq:leastsquares} and only take the first $T_1$ Markov parameters from its solution; the second method is to use the OLS estimator~\eqref{eq:leastsquares_v2}. Note that in the presence of no process noise $w_t = 0$, we can see that there is always an extra noise term $E^{(i)}$ in~\eqref{eq:outputi_v2}. This means the  OLS estimator~\eqref{eq:leastsquares_v2} is not as good as the OLS estimator~\eqref{eq:leastsquares} in this case. This is confirmed by the numerical results shown in Figure~\ref{fig:stable_varyingT1T2}. When there is process noise,  the OLS estimators~\eqref{eq:leastsquares} and~\eqref{eq:leastsquares_v2} have different balance: the constant $F$ is increasing in~\eqref{eq:leastsquares}, while the extra noise term $E^{(i)}$ in~\eqref{eq:leastsquares_v2} is growing, when increasing the rollout length $T_2$.  Both these two OLS estimators suffer from the system instability if there is process noise. In  Figure~\ref{fig:stable_varyingT1T2}, the results from the OLS estimator~\eqref{eq:leastsquares} is numerically better than the OLS estimator~\eqref{eq:leastsquares_v2} for this case.

\begin{figure*}[t]
\setlength{\abovecaptionskip}{2pt}
    \centerfloat
	{
   \vspace{-10mm} \includegraphics[scale=0.7]{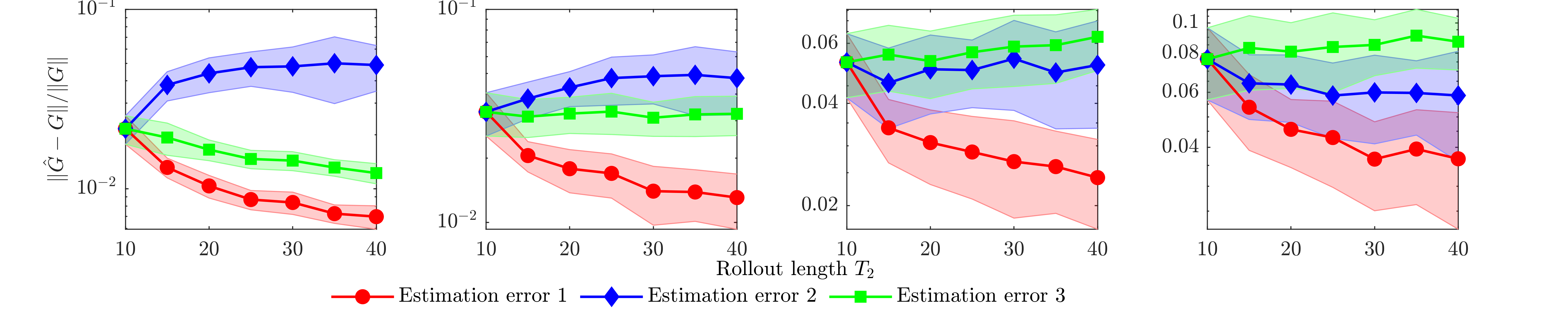}
    }
    \caption{Estimation error for system~\eqref{eq:unstable} with fixed $T_1 = 10$. We varied rollout length $T_2$ and process noises. Left to right: $\sigma_w = 0$; $\sigma_w = 0.2$; $\sigma_w = 0.4$; $\sigma_w = 0.6$. Estimation error 1: the result from OLS estimator~\eqref{eq:leastsquares}, and only take the first $T_1$ Markov parameters; Estimation error 2: the result from OLS estimator~\eqref{eq:leastsquares_v2}; Estimation error 3: the result from OLS estimator~\eqref{eq:leastsquares} and compare it with $T_2$ Markov parameters.}
    \label{fig:stable_varyingT1T2}
 \end{figure*}
\vspace{5mm}

\bibliographystyle{ieeetr}
\bibliography{references}

\end{document}